\documentclass{amsart}

\usepackage{amssymb,amsfonts,amsxtra}
\usepackage{dsfont,mathrsfs}
\renewcommand{\mathcal}{\mathscr}
\usepackage[all]{xypic}
\xyoption{dvips}
\usepackage{url}
\usepackage[colorlinks,plainpages,backref]{hyperref}
\usepackage[neveradjust]{paralist}
\usepackage{ifthen}

\theoremstyle{definition}
\newtheorem{ntn}{Notation}[section]
\newtheorem{dfn}[ntn]{Definition}
\theoremstyle{plain}
\newtheorem{lem}[ntn]{Lemma}
\newtheorem{prp}[ntn]{Proposition}
\newtheorem{thm}[ntn]{Theorem}
\newtheorem{cor}[ntn]{Corollary}

\theoremstyle{remark}

\newtheorem{rmk}[ntn]{Remark}
\newtheorem{exa}[ntn]{Example}
\numberwithin{equation}{section}

\newcommand{\ideal}[1]{{\left\langle#1\right\rangle}}

\newcommand{\into}{\hookrightarrow}
\newcommand{\onto}{\twoheadrightarrow}
\newcommand{\p}{\partial}
\newcommand{\dt}[1][]{\ifthenelse{\equal{#1}{}}{\p_t}{\p_{t_{#1}}}}

\newcommand{\mm}{\mathfrak{m}}
\newcommand{\pp}{\mathfrak{p}}

\newcommand{\ZZ}{\mathds{Z}}
\newcommand{\NN}{\mathds{N}}

\newcommand{\eps}{\varepsilon}
\newcommand{\wt}{\widetilde}

\DeclareMathOperator{\ch}{char}
\DeclareMathOperator{\Der}{Der}
\DeclareMathOperator{\End}{End}

\DeclareMathOperator{\Hom}{Hom}

\begin{document}

\title[Quasihomogeneity of curves]{Quasihomogeneity of curves\\and the Jacobian endomorphism ring}

\author{Michel Granger}
\address{
M.~Granger\\
Universit\'e d'Angers, D\'epartement de Math\'ematiques\\ 
LAREMA, CNRS UMR n\textsuperscript{o}6093\\ 
2 Bd Lavoisier\\ 
49045 Angers\\ 
France
}
\email{granger@univ-angers.fr}

\author{Mathias Schulze}
\address{
M.~Schulze\\
Department of Mathematics\\
University of Kaiserslautern\\
67663 Kaiserslautern\\
Germany}
\email{mschulze@mathematik.uni-kl.de}
\thanks{The research leading to these results has received funding from the People Programme (Marie Curie Actions) of the European Union's Seventh Framework Programme (FP7/2007-2013) under REA grant agreement n\textsuperscript{o} PCIG12-GA-2012-334355.}

\date{\today}



\begin{abstract}
We give a quasihomogeneity criterion for Gorenstein curves.
For complete intersections, it is related to the first step of Vasconcelos' normalization algorithm.
In the process, we give a simplified proof of the Kunz--Ruppert criterion.
\end{abstract}

\maketitle
\tableofcontents

\section{Introduction}

We consider a reduced algebroid curve $X$ over an algebraically closed field $k$ of characteristic $0$ with coordinate ring $A$.
The Jacobian ideal of $X$ is the $1$st Fitting ideal $J_A:=F^1_A(\Omega^1_A)$, where $\Omega^1_A$ is the universally finite derivation of $A$. 
Lipman~\cite{Lip69} showed that $X$ is smooth if and only if $J_A$ is principal.
Based on this equivalence, Vasconcelos~\cite{Vas91} showed that the normalization of $X$ is obtained by repeatedly replacing $A$ by the endomorphism ring 
\[
\End_A(J_A^{-1})=(J_AJ_A^{-1})^{-1}.
\]
Not much is known about these endomorphism rings.
As a coarse measure for the ``amount of normality'' achieved by this operation we consider the length
\[
\rho_X:=\ell(\End_A(J_A^{-1})/A).
\]
Smoothness of $X$ is equivalent to $\rho_X=0$ and it is natural to ask when $\rho_X=1$.
We shall answer this question for complete intersection curves.

By definition, $X$ is \emph{quasihomogeneous} if the kernel of some epimorphism 
\[
k[[x_1,\dots,x_n]]\onto A
\]
is a quasihomogeneous ideal with respect to some positive weights on the variables.

The main result of this article is the following quasihomogeneity criterion.

\begin{thm}\label{0}
Let $X$ be a non-smooth complete intersection algebroid curve over a field $k=\bar k$ with $\ch k=0$.
Then $X$ is quasihomogeneous if and only if $\rho_X=1$.
\end{thm}

The proof of Theorem~\ref{0} follows from Theorem~\ref{30} and Corollary~\ref{26} which rely on a study of semigroups of curves developed in Sections~\ref{14} and \ref{15}.

\section{Fractional ideals}\label{13}

Let $\pp_1,\dots,\pp_r$ be the (minimal) associated primes of $A$.
Then $A_i:=A/\pp_i$, $i=1,\dots,r$, are the coordinate rings of the branches of $X$.
By reducedness of $A$, \cite[Cor.~2.1.13]{HS06}, Serre's normality criterion and Cohen's structure theorem, the normalization $\tilde A$ of $A$ in the total ring of fractions $L:=Q(A)$ factorizes as 
\begin{equation}\label{11}
A\into\prod_{i=1}^rA_i\into\prod_{i=1}^r\tilde A_i=\tilde A\into\prod_{i=1}^rL_i=L,\quad\tilde A_i=k[[t_i]],\quad L_i=Q(A_i)=Q(\tilde A_i)
\end{equation}
and we identify $A$ and $\tilde A$ with their images in $L$.
For $\alpha=(\alpha_1,\dots,\alpha_r)\in\ZZ^r$, we shall abbreviate $t^\alpha:=(t_1^{\alpha_1},\dots,t_r^{\alpha_r})$, $\dt:=(\dt[1],\dots,\dt[r])$, and $t^\alpha\dt:=(t_1^{\alpha_1}\dt[1],\dots,t_r^{\alpha_r}\dt[r])$.

Recall that a \emph{fractional ideal} is a finite $A$-submodule of $L$ containing a non-zero divisor of $A$. 
For any $A$-module $I$, we denote the dual by
\[
I^{-1}:=\Hom_A(I,A)
\] 
It is easy to prove the following well-known statement.

\begin{lem}\label{3}
For each two fractional ideals $I_1$ and $I_2$, also 
\begin{equation}\label{4}
\Hom_A(I_1,I_2)=\{x\in L\mid xI_1\subseteq I_2\}
\end{equation}
is again a fractional ideal.
In particular, this applies to the dual $I^{-1}$.
Moreover, $I\mapsto I^{-1}$ is inclusion reversing.
\end{lem}  

\begin{exa}\label{10}\
\begin{asparaenum}

\item\label{10a} The maximal ideal $\mm_A$ and the Jacobian ideal $J_A$ of $A$ and the normalization $\tilde A$ are fractional ideals.

\item\label{10b} The conductor 
\[
C_A:=\tilde A^{-1}
\]
is a fractional ideal by Lemma~\ref{3} and \eqref{10a}.
It is the largest ideal of $\tilde A$ which is also an ideal of $A$.
Since $\tilde A$ is a product of principal ideal domains and $A$ is not smooth, 
\begin{equation}\label{20}
C_A=\ideal{t^\delta}\subseteq\mm_A
\end{equation}
is generated by a monomial.

\item\label{10c} The module 
\[
M_A:=At\dt\mm_A
\]
is a fractional ideal that is related to quasihomogeneity of $X$.

\item\label{10d} For any fractional ideal $M$, $\End_A(M)$ is a fractional ideal and 
\[
A\subseteq\End_A(M)\subseteq\tilde A
\]
since the characteristic polynomial of an endomorphism is a relation of integral dependence by the Cayley--Hamilton theorem. 

\end{asparaenum}
\end{exa}

We shall see that $J_A$ is related to $M_A$ and define
\[
\rho_X':=\ell(\End_A(M_A^{-1})/A).
\]

\begin{rmk}
Note that for $X$ smooth we have $\rho_X=0=\rho_X'$.
\end{rmk}

From now on we assume that $X$ is not smooth.

\begin{lem}\label{2}
For any non-principal fractional ideal $I$, we have $I^{-1}=\Hom_A(I,\mm_A)$ as fractional ideals.
In particular, $\mm_A^{-1}=\End_A(\mm_A)$.
\end{lem}

\begin{proof}
After multiplying by a unit in $L$, we may assume that $I\subseteq A$.
Any surjection $\phi\colon I\onto A$ would have to split as $I=A\cdot x\oplus I'$ with $x\in I$ a non-zero divisor and $I'$ an $A$-module.
For $x'\in I'$, we have $x'x\in Ax\cap I'=0$ and hence $x'=0$.
But then $I$ would be principal contradicting to the assumption.
Thus, any $\phi\colon I\to A$ must map to $\mm_A$ and the first claim follows.
The second claim is due our assumption that $X$ is not smooth. 
\end{proof}

We denote by $-^\vee:=\Hom_A(-,\omega_A^1)$ the dualizing functor.
From now on we assume that $A$ is Gorenstein.
Then $\omega_A^1\cong A$ and hence 
\begin{equation}\label{5}
-^{-1}\cong-^\vee
\end{equation}
is an involution on fractional ideals.

\begin{lem}\label{7}
For every fractional ideal $I$, we have $\End_A(I)=\End_A(I^{-1})$ as fractional ideals.
\end{lem}

\begin{proof}
By \eqref{4}, any $\phi\in\End_A(I)$ is just multiplication by some $x\in L$.
The same $x$ corresponds to $\phi^{-1}\in\End_A(I^{-1})$ and hence $\End_A(I)\subseteq\End_A(I^{-1})$.
By \eqref{5}, the claim follows by applying the above argument to $I^{-1}$ instead of $I$.
\end{proof}

\section{Quasihomogeneity of curves}\label{31}

In the following theorem, we summarize several versions of the Kunz--Ruppert criterion for quasihomogeneity of curves.
The original formulation is the equivalence \eqref{9a} $\Leftrightarrow$ \eqref{9d} $\Leftrightarrow$ \eqref{9b}.
In the appendix, we comment on a possible issue in its proof by Kunz--Ruppert~\cite{KR77} and give a simplified argument.

The equivalence with \eqref{9c} originates from work of Greuel--Martin--Pfister~\cite[Satz~2.1]{GMP85} extending the criterion by a numerical characterization of quasihomogeneity, in case of Gorenstein curves.
The implication (5) $\Rightarrow$ (6) in their main result was generalized by Kunz--Waldi~\cite[Thm.~6.21]{KW88} by comparing two modules:
\begin{asparaenum}
\item The module of \emph{Zariski differentials}, which is the reflexive hull $((\Omega_A^1)^{-1})^{-1}$ of $\Omega_A^1$.
By the universal property of $\Omega_A^1$, $(\Omega_A^1)^{-1}=\Der_k(A)$ and hence
\begin{equation}\label{6}
((\Omega_A^1)^{-1})^{-1}=\Der_k(A)^{-1}=(\Hom_A(A\dt A,A)\dt)^{-1}=((A\dt\mm_A)^{-1}\dt)^{-1}=M_A\frac{dt}{t}
\end{equation}
\item The module of \emph{exact differentials} $c_AdA=\dt\mm_Adt$ where
\begin{equation}\label{29}
c_A\colon\Omega_A^\bullet\to\omega_A^\bullet
\end{equation}
is the \emph{trace map} into the \emph{regular differential forms} on $A$, which are certain meromorphic forms satisfying $\omega_A^\bullet\cong(\Omega_A^{1-\bullet})^\vee$ (see \cite{Ker84,KW88}).
\end{asparaenum}

\begin{thm}[Kunz--Ruppert--Waldi]\label{9}
The following statements are equivalent:
\begin{asparaenum}
\item\label{9a} The curve $X$ is quasihomogeneous.
\item\label{9d} For some derivation $\chi\in\Der_k(A)$, $A\cdot\chi(A)=\mm_A$.
\item\label{9b} Multiplication by some unit in $\tilde A$ induces an isomorphism $\mm_A\cong t\dt\mm_A$. 
\item\label{9c} Every Zariski differential is exact, that is, $t\dt\mm_A=M_A$.
\end{asparaenum}
\end{thm}

\begin{cor}\label{28}
If $X$ is quasihomogeneous then $\rho_X'=1$. 
\end{cor}

\begin{proof}
By Theorem~\ref{9}, we have $\End_A(M_A)=\End_A(\mm_A)$ as fractional ideals and hence $\rho_X'=\ell(\mm_A^{-1}/A)=1$ by Lemma~\ref{2} and the Gorenstein hypothesis.
\end{proof}

\begin{lem}\label{8}
We have an inclusion $\End_A(\mm_A)\subseteq\End_A(M_A)$.
\end{lem}

\begin{proof}
By \eqref{6} and Lemma~\ref{7}, $\Der_k(A)=M_A^{-1}\partial_t$ and hence $\End_A(M_A)=\End_A(\Der_k(A))$.
So, by Lemma~\ref{2}, it suffices to show that $\mm_A^{-1}\subseteq\End_A(\Der_k(A))$.
But any $\delta\in\Der_k(A)$ lifts uniquely to $\delta'\in\Der_k(L)$ and such a $\delta'$ is in $\Der_k(A)$ exactly if $\delta'(A)\subseteq A$.
Now, let $\phi\in\mm_A^{-1}$ be multiplication by $x\in L$.
Then $x\delta'(A)\subseteq x\mm_A\subseteq A$ since $\delta'(k)=0$.
The claim follows.
\end{proof}

By Lemma~\ref{8}, we have the following chain of fractional ideals
\[
A\subseteq\End_A(\mm_A)\subseteq\End_A(M_A)
\]
where the colength of the first inclusion equals $\ell(\mm_A^{-1}/A)=1$ by Lemma~\ref{2} and the Gorenstein hypothesis. 
This yields the following

\begin{prp}\label{12}
$\rho_X'=1$ implies that $\End_A(\mm_A)=\End_A(M_A)$.
\end{prp}

In Propositions~\ref{16} and \ref{17} in the following sections, we shall see that the equality in Proposition~\ref{12} implies that in Theorem~\ref{9}.\eqref{9c}. 
Combined with Corollary~\ref{28} this proves the following statement.

\begin{thm}\label{30}
A Gorenstein curve $X$ is quasihomogeneous if and only if $\rho'_X=1$.
\end{thm}

Recall the following result from \cite{Sch13} which uses the coincidence of the Jacobian and the $\omega$-Jacobian ideal for complete intersections (see \cite[\S3]{OZ87}).

\begin{prp}\label{1}
If $X$ is a complete intersection then $\omega^0_A\cong J_A^{-1}$.
\end{prp}

\begin{proof}
By hypothesis and \cite[Prop.~1]{Pie79}, \eqref{29} induces a surjection $\Omega_A^1\onto J_A$ with torsion kernel.
The claim follows by dualizing.
\end{proof}

In the situation of Proposition~\ref{1}, \eqref{6} yields
\[
J_A\cong(\omega_A^0)^{-1}\cong((\Omega_A^1)^\vee)^{-1}\cong((\Omega_A^1)^{-1})^{-1}= M_A
\]
and, by Lemma~\ref{7}, we deduce the following statement.

\begin{cor}\label{26}
If $X$ is a complete intersection curve then $\End_A(J_A)=\End_A(M_A)$ and, in particular, $\rho_X=\rho_X'$.
\end{cor}

\section{Semigroups}\label{14}

Let $\nu_i\colon L_i\to\ZZ\cup\{\infty\}$ be the discrete valuation with respect to the parameter $t_i$ and define the \emph{multivaluation} on $L$ to be 
\[
\nu:=(\nu_1,\dots,\nu_r):L\to(\ZZ\cup\{\infty\})^r.
\]
Let $D(L):=\{x\in L\mid \forall i=1,\dots,r: x_i\ne0\}$ denote the set of non-zero divisors in $L$ and set $D(M):=D(L)\cap M$ for any subset $M\subseteq L$.
Note that $D(A):=A\setminus\bigcup_{i=1}^r\pp_i$.

\begin{dfn} 
For any subset $M\subseteq L$, we set $\Gamma(M):=\nu(M\cap D(L))$.
Then the \emph{semigroup of $A$} is defined as
\[
\Gamma_A:=\Gamma(A)\subset\NN^r.
\]
\end{dfn}

Note that, for any fractional ideal $I$, $\Gamma(I)$ is a \emph{$\Gamma_A$-set}, that is, 
\[
\alpha\in\Gamma_A,\beta\in\Gamma(I)\Rightarrow \alpha+\beta\in\Gamma(I). 
\]

Although, in general, $\mm_A$ and $t\dt\mm_A$ are incomparable and $\mm_A\not\cong M_A$ (see appendix), we have at least $t\dt\mm_A\subseteq M_A$ and
\begin{equation}\label{19}
\Gamma(\mm_A)=\Gamma(t\dt\mm_A)\subseteq\Gamma(M_A).
\end{equation}
The purpose of this section is to prove the following result.

\begin{prp}\label{16}
If $\Gamma(\mm_A)=\Gamma(M_A)$ then $t\dt\mm_A=M_A$.
\end{prp}

By \eqref{20}, we have that
\[
C_A\subseteq t\dt\mm_A
\]
Thus, the proof of Proposition~\ref{16} follows from equation \eqref{19} and the following lemma.

\begin{lem}\label{21}
Let $M\subseteq N\subseteq L$ be two k-vector subspaces of $\tilde A$ and $\delta\in\NN_+^r$ such that $t^\delta\tilde A\subseteq M$. 
Then the equality $\Gamma(M)=\Gamma(N)$ implies $M=N$.
\end{lem}

\begin{proof}
Let $x=(x_1,\dots,x_r)$ be an element of $N$. 
If $\nu_i(x)\geq\delta_i$ for all $i$, then already $x\in t^\delta\tilde A\subseteq M$. 
 
We can eliminate the zero components of $x$ by choosing $y\in t^\delta\tilde A\subseteq M$ such that $x+y\in D(N)$ and hence $\nu_i(x+y)\geq\min(\nu_i(x),\delta_i)$ for all $i$. 
If for some $i$ we have $\nu_i(x)<\delta_i$, there is by hypothesis an element $x'\in D(M)$ such that $\nu(x')=\nu(x+y)$ and $\nu_i(x+y-x')>\nu_i(x+y)$ as well as $\nu_j(x+y-x')\geq\nu_j(x+y)$ for all $j\neq i$. 
We can then conclude by an induction on $\sum_{i=0}^r\max\{\delta_i-\nu_i(x),0\}$ that  $x+y-x'\in M$ and hence $x\in M$.
\end{proof}

\begin{rmk}
By Lemma~\ref{21}, $\delta'+\NN^r\subseteq\Gamma_A$ implies that $\delta\le\delta'$ for $\delta$ as in \eqref{20}.
\end{rmk}

\section{Gorenstein symmetry}\label{15}

The purpose of this section is to prove the following result.

\begin{prp}\label{17}
If $\End_A(\mm_A)=\End_A(M_A)$ then $\Gamma(\mm_A)=\Gamma(M_A)$.
\end{prp}

The proof of Proposition~\ref{17} uses that $X$ being Gorenstein is equivalent to a symmetry property of $\Gamma_A$ which is due to Kunz~\cite{Kun70} in the irreducible case and to Delgado~\cite{Del88} in general. 
The formulation of the precise statement requires some notation.
Recall that $\Gamma(C_A)=\delta+\NN^r$ by \eqref{20} and we set $\tau=\delta-(1,\dots,1)$. 
For any $\alpha\in\ZZ^r$, we denote 
\[
\Delta(\alpha):=\bigcup_{i=1}^r\Delta_i(\alpha),\quad \Delta_i(\alpha):=\{\beta\in\ZZ^r\mid \alpha_i=\beta_i \text{ and } \alpha_j<\beta_j \text{ if } j\neq i\}.
\]

\begin{dfn}
The semigroup $\Gamma_A$ is called \emph{symmetric} if 
\begin{equation}\label{22}
\forall\alpha\in\ZZ^r:\alpha\in\Gamma_A\Leftrightarrow\Delta(\tau-\alpha)\cap\Gamma_A=\emptyset.
\end{equation}
\end{dfn}

\begin{thm}(Delgado)
The curve $X$ is Gorenstein if and only if its semigroup $\Gamma_A$ is symmetric.
\end{thm}

\begin{rmk}
In the irreducible case $r=1$ the symmetry condition of Delgado reduces to the classical Kunz symmetry condition
\[
\forall\alpha\in\{0,\dots,\tau\}:\alpha\in\Gamma_A\Leftrightarrow\tau-\alpha\notin\Gamma_A.
\]
\end{rmk}

We prove Proposition~\ref{17} in a sequence of lemmas.

\begin{lem}\label{36}\
\begin{asparaenum}
\item\label{36a} $\Gamma_A\subseteq\{0\}\cup\left((1,\dots,1)+\NN^r\right)$.
\item\label{36b} $\tau\in\Gamma_A$ if and only if $r>1$.
\item\label{36c} $\Delta(\tau)\cap\Gamma_A=\emptyset$.
\end{asparaenum}
\end{lem}

\begin{proof}\
\begin{asparaenum}

\item Let $\alpha\in\Gamma_A$ be such that $\alpha_i=0$ for some $i$.
Then $\alpha=\nu(x)$ for some $x\in A$ with $x_i\not\in\mm_{A_i}$. 
This implies that $x\not\in\mm_A$ and hence $\alpha=0$.

\item By \eqref{36a}, $\Delta(0)\cap\Gamma_A=\emptyset$ if and only if $r>1$ and the claim follows from \eqref{22}.

\item This follows from $0\in\Gamma_A$ and \eqref{22}.

\end{asparaenum}
\end{proof}

\begin{rmk}
For any $\beta\in\NN^r$, $\tau+\beta\notin\Gamma_A$ if and only if $\beta$ has exactly one zero component, generalizing Lemma~\ref{36}.\eqref{36c}.
\end{rmk}

\begin{lem}\label{24}\
\begin{asparaenum}
\item\label{24a} If $\Gamma(\mm_A)\subsetneq\Gamma(M_A)$ then $\Delta(\tau)\cap\Gamma(M_A)\ne\emptyset$.
\item\label{24b} If $\Delta(\tau)\cap\Gamma(M_A)\ne\emptyset$ then $\Delta_i(\tau)\subseteq\Gamma(M_A)$ for some $i$.
\end{asparaenum}
\end{lem}

\begin{proof}\
\begin{asparaenum}

\item Let $\alpha\in\Gamma(M_A)\setminus\Gamma(\mm_A)$.
Then, by \eqref{22} and \eqref{19}, there is a $\beta\in\Delta_i(\tau-\alpha)\cap\Gamma_A\subseteq\Gamma(\mm_A)\subseteq\Gamma(M_A)$.
Thus, $\alpha+\beta\in\Delta_i(\tau)\cap\Gamma(M_A)\subseteq\Delta(\tau)\cap\Gamma(M_A)$.

\item We may assume that there is an element $x\in M_A$ with $\nu(x)\in\Delta_1(\tau)\cap\Gamma(M_A)$.
Up to a factor in $k^*$, $x\equiv(t^{\tau_1},\dots)\mod t^\delta\tilde A$.
For any $\beta\in\Delta_1(\tau)$, $x-t^\beta\in t^\delta\tilde A\subseteq M_A$ and hence $t^\beta\in M_A$ and $\beta\in\Gamma(M_A)$.

\end{asparaenum}
\end{proof}

\begin{lem}\label{25}\
\begin{asparaenum}
\item\label{25a} $\Gamma(\End_A(\mm_A))\setminus\Gamma_A=\Delta(\tau)$.
\item\label{25b} If $\Delta_i(\tau)\subseteq\Gamma(M_A)$ for some $i$ then $(\Gamma(\End_A(M_A))\setminus\Gamma_A)\cap\bigcup_{j<0}je_i+\Delta_i(\tau)\neq\emptyset$.
\end{asparaenum}
\end{lem}

\begin{proof}\
\begin{asparaenum}

\item By Lemma~\ref{36}.\eqref{36a}, $\Delta(\tau)+\Gamma(\mm_A)\subseteq\delta+\ZZ^r\subseteq\Gamma(\mm_A)$ and hence $\supseteq$ by Lemma~\ref{21} and \ref{36}.\eqref{36c}.
To prove $\subseteq$, let $\alpha\in\Gamma(\End_A(\mm_A))\setminus\Gamma_A$.
Then, by \eqref{22}, there is a $\beta\in\Delta(\tau-\alpha)\cap\Gamma_A$ and hence $\alpha+\beta\in\Delta(\tau)$.
As $\beta\in\Gamma(\mm_A)$ leads to the contradiction $\alpha+\beta\in\Delta(\tau)\cap\Gamma(\mm_A)=\emptyset$ by Lemma~\ref{36}.\eqref{36c}, we must have $\beta=0$ and hence $\alpha\in\Delta(\tau)$.

\item There exists a minimal $m\le0$ such that $\bigcup_{m<j\le1}je_i+\Delta_i(\tau)\subseteq\Gamma(\mm_A)$.
In fact, $m\ge-\tau_i$ by Lemma~\ref{36}.\eqref{36a}.
By \eqref{19} and the hypothesis, $\delta+me_i+\NN^r\subseteq\Gamma(M_A)$ and hence $\alpha+\Gamma(M_A)\subseteq\Gamma(\mm_A)\subseteq\Gamma(M_A)$ for any $\alpha\in me_i+\Delta_i(\tau)\setminus\Gamma_A$ by Lemma~\ref{36}.\eqref{36a}.
This implies $t^\alpha\in\End_A(M_A)$ by Lemma~\ref{21}.

\end{asparaenum}
\end{proof}

\begin{proof}[Proof of Proposition~\ref{17}]
Assuming that $\Gamma(\mm_A)\subsetneq\Gamma(M_A)$, Lemma~\ref{24} applies followed by Lemma~\ref{25}.
The conclusions of the latter show that 
\[
\Gamma(\End_A(\mm_A))\subsetneq\Gamma(\End_A(M_A))
\]
and hence that $\End_A(\mm_A)\subsetneq\End_A(M_A)$ by Lemma~\ref{21}.
\end{proof}

\section*{Appendix: Kunz-Ruppert criterion}\label{32}

In the process of proving the implication \eqref{9d} $\Rightarrow$ \eqref{9a} in Theorem~\ref{9}, Kunz and Ruppert seem to claim (see \cite[page~6, line~2]{KR77}) and use that
\begin{equation}\label{33}
A\cdot t\p_t(A)\cong\mm_A
\end{equation}
as $A$-submodule of $\tilde A$.
The following is a counter-example for this statement.

By abuse of notation, we denote $\mm_{\wt A}:=\mm_{\wt A_1}\times\cdots\times\mm_{\wt A_r}$.

\begin{exa}
Consider the (non-quasihomogeneous) plane curve singularity defined by $x^4+xy^4+y^5=0$. 
After a coordinate change, the equation reads
\[
f:=x^4-y(x+y)^4=0.
\]
Then the normalization $A=k[[x,y]]/\ideal{f}\subset\wt A=k[[t]]$ is given by
\[
x=\frac{t^5}{1-t}=t^5+t^6+t^7+\cdots,\quad y=t^4.
\]
On the other hand, the left hand side of \eqref{33} considered modulo $\mm_{\wt A}^8\supset\mm_A\cdot t\p_t(A)$ is the $k$-vector space generated by the $7$-jets
\begin{equation}\label{23}
\wt x=t\p_t(x)=\eta\cdot x\equiv5t^5+6t^6+7t^7\mod\mm_{\wt A}^8,\quad\wt y=t\p_t(y)=4y,\quad
\end{equation}
where
\[
\eta:=\frac{5-4t}{1-t}.
\]
If there were an isomorphism \eqref{33} then, both sides being fractional ideals, it would have to be induced by multiplication by some unit $\eps\in\wt A^*$ with
\[
\eps\equiv\eta\equiv5+t+t^2\mod\mm_{\wt A}^3.
\]
Note that the $3$-jet
\[
\eps\equiv 5+t+t^2+\alpha t^3\mod\mm_{\wt A}^4,\quad\alpha\in k,
\]
determines the $7$-jet
\[
\eps\cdot y\equiv5t^4+t^5+t^6+\alpha t^7\mod\mm_{\wt A}^8.
\]
But for no choice of $\alpha$ this expression lies in the $k$-span of \eqref{23}.
Therefore, there is no isomorphism \eqref{33} for the curve under consideration.
\end{exa}

The following Proposition~\ref{34} contains the statement of \cite[Satz 2.2]{KR77}, which yields the implication \eqref{9d} $\Rightarrow$ \eqref{9a} in Theorem~\ref{9}.

\begin{rmk}
Let $\chi\in\Der_k(A)$.
By Scheja--Wiebe~\cite[(2.5)]{SW73}, $\chi(\pp_i)\subset\pp_i$ and hence $\chi$ induces a derivation $\chi_i\in\Der_k(A_i)$.
As $A_i$ is a domain, Seidenberg~\cite{Sei66} shows that $\chi_i$ lifts to a derivation $\wt\chi_i\in\Der_k(\wt A_i)$.
So by \eqref{11}, $\wt\chi:=(\wt\chi_1,\dots,\wt\chi_r)\in\Der_k(\wt A)$ is a lift of $\chi$.
As $\chi$ extend uniquely to any localization and hence to $L$, $\wt\chi$ is unique.
While this proves part 1) of \cite[Satz 2.2]{KR77}, it is actually not needed.
\end{rmk}

Recall \cite[page 168, Def.]{SW73}, that a derivation $\delta\in\Der_k(A)$ is called diagonalizable if $\mm_A$ is generated by eigenvectors of $\delta$.

\begin{prp}\label{34}
Any $\chi\in\Der_k(A)$ satisfying \eqref{9d} in Theorem~\ref{9} lifts uniquely to $\wt\chi\in\Der_k(\wt A)$ such that
\begin{equation}\label{37}
\wt\chi=\gamma\cdot t\p_t
\end{equation}
for some 
\begin{equation}\label{40}
\gamma\in k^r
\end{equation}
after a suitable coordinate change.
Moreover, $\chi$ is diagonalizable with non-zero eigenvalues on $\mm_A$ and can be chosen with eigenvalues in $\NN_+$.
\end{prp}

\begin{proof}
The $k$-derivation $\chi$ lifts uniquely to a $k$-derivation $\wt\chi=(\wt\chi_1,\dots,\wt\chi_r)\in\Der_k(L)$.
By finiteness of the normalization, $0\ne x_i\in\mm_{\wt A_i}$ for some $x=(x_1,\dots,x_r)\in\mm_A$.
Choosing $x\in\mm_A$ with $\nu_i(x_i)$ minimal yields $\nu_i(\wt\chi_i(x_i))\ge\nu_i(x_i)$.
By \eqref{9d} in Theorem~\ref{9}, equality holds for some such choice of $x\in\mm_A$.
Hence $\wt\chi_i=\gamma_i\cdot t_i\p_{t_i}$ for some $\gamma_i\in\wt A_i^*$ and \eqref{37} is obtained by setting $\gamma:=(\gamma_1,\dots,\gamma_r)\in\wt A^*$.

In order to achieve \eqref{40} by a coordinate change, 
apply the Poincar\'e--Dulac decomposition theorem (see \cite[Ch.~3.~\S3.2]{AA88} or \cite[Satz~3]{Sai71}) to $\delta=\wt\chi_i$.
Then
\[
\delta=\sigma+\eta,\quad\sigma=\gamma(0)\cdot t\p_t,\quad\eta\in\End_k(\mm_{\wt A_i}/\mm_{\wt A_i}^2)\text{ nilpotent},\quad [\sigma,\eta]=0,
\]
and hence $\eta=0$.
Finally, the conductor is a $\wt\chi$-invariant (see \eqref{20}) finite-codimensional $k$-subspace of $\wt A$ contained in $\mm_A$ which yields the last statement (see \cite[page 7]{KR77} for details).
\end{proof}

\bibliographystyle{amsalpha}
\bibliography{qhcend}
\end{document}